\documentclass{amsart}
\usepackage[dvipdfmx]{hyperref}
\usepackage{tikz}
\usepackage{amsmath,amssymb}
\usepackage{amsthm}
\usepackage{color}
\usepackage{mathrsfs}
\usepackage{bm}
\usepackage{mathdots}
\usepackage{comment}
\usepackage[all]{xy}

\numberwithin{equation}{section}

\usetikzlibrary{arrows,calc,positioning}

\title{BPHZ Renormalization in Gaussian Rough Paths}
\author[H.~Ito]{Hayahide Ito}
\address{Graduate~School~of~Engineering~Science, Osaka~University, Osaka~560-8531, Japan}
\email{h-ito@sigmath.es.osaka-u.ac.jp}
\subjclass[2020]{Primary~60L30, Secondary~60L20, 91-10}

\newtheoremstyle{claim}
  {\topsep}
  {\topsep}
  {}
  {}
  {}
  {}
  {.5em}
  {\thmname{#1}\thmnumber{ #2}\thmnote{ (#3)}}

\theoremstyle{theorem}
\newtheorem{theorem}{Theorem}[section]
\newtheorem{prop}[theorem]{Proposition}
\newtheorem{lemma}[theorem]{Lemma}

\theoremstyle{remark}
\newtheorem{example}[theorem]{Example}
\newtheorem{definition}[theorem]{Definition}
\newtheorem{remark}[theorem]{Remark}

\newcommand{\I}{\mathcal{I}}

\renewcommand{\L}{\mathfrak{L}}
\newcommand{\Type}{{\bf Type}}
\newcommand{\supp}{\mathrm{supp}\,\,}
\newcommand{\delexminus}{\Delta_-^{\mathrm{ex}}}
\newcommand{\delexplus}{\Delta_+^{\mathrm{ex}}}

\newcommand{\pexminus}{\mathfrak{p}^{\mathrm{ex}}_-}
\newcommand{\pexplus}{\mathfrak{p}^{\mathrm{ex}}_+}

\newcommand{\taexminus}{\tilde{\mathcal{A}}_-^{\mathrm{ex}}}

\newcommand{\gminus}{g^-({\bf \Pi})}
\newcommand{\htprominus}{\hat{\mathcal{M}}_-^{\mathrm{ex}}}
\newcommand{\iexminus}{\mathfrak{i}_-^{\mathrm{ex}}}
\newcommand{\ve}{\varepsilon}
\newcommand{\hxi}{\hat{\Xi}}

\newcommand{\hpv}{\hat{\Pi}^{\ve}}

\makeatletter
\newcommand{\opnorm}{\@ifstar\@opnorms\@opnorm}
\newcommand{\@opnorms}[1]{%
  \left|\mkern-1.5mu\left|\mkern-1.5mu\left|
   #1
  \right|\mkern-1.5mu\right|\mkern-1.5mu\right|
}
\newcommand{\@opnorm}[2][]{%
  \mathopen{#1|\mkern-1.5mu#1|\mkern-1.5mu#1|}
  #2
  \mathclose{#1|\mkern-1.5mu#1|\mkern-1.5mu#1|}
}
\makeatother

\begin{document}
\begin{abstract}
We construct a procedure for Bogoliubov--Parasiuk--Hepp--Zimmermann (BPHZ) renormalization of a rough path in view of the relation between rough path theory and regularity structure.  We also provide a plain expression of the BPHZ-renormalized model in a rough path.  BPHZ renormalization plays a central role in the theory of singular stochastic partial differential equations and assures the convergence of the model in the regularity structure. Here we demonstrate that the renormalization is also effective in a rough path setting by considering its application to rough path theory and mathematical finance.
\end{abstract}
\maketitle

\section{Introduction}
We consider the following It\^{o}-type stochastic integral: for $T>0$ and a smooth function $f: \mathbb{R} \to \mathbb{R}$,
\begin{equation}
	\label{eq:introRough}
	\int_0^T f(W_t^H) dW_t,
\end{equation}
where $W_t$ is a standard one-dimensional Brownian motion, and $W_t^H$ is a fractional Brownian motion with Hurst parameter $H \in (0, 1/2)$ given by
\[
	W_t^H = \sqrt{2H} \int_0^t |t-s|^{H - 1/2} dW_s.
\]
Integral (\ref{eq:introRough}) emerges in option pricing in a financial market with rough volatility (see~\cite{RoughVolRS}), so the Wong--Zakai approximation of (\ref{eq:introRough}) is valuable for financial mathematics.  However, for any mollifier $\varrho: \mathbb{R} \to \mathbb{R}$, the limit
\[	
	\lim_{\ve \to 0} \int_0^T f(W_t^{H, \ve}) dW_t^{\ve}
\]
does not exist in general, where $\varrho^{\ve}(\cdot) = \varrho(\cdot / \ve)/ \ve$, $W_t^{H, \ve} = (\varrho^{\ve} \ast W^H)_t$, and $W_t^{\ve} = (\varrho^{\ve} \ast W)_t$.  This failure is due to fractional Brownian motion $W_t^H$ having lower regularity than Brownian motion $W_t$.  Bayer et~al.~\cite{RoughVolRS} demonstrated the following revised Wong--Zakai approximation by adding a term to prevent the divergence.
\begin{theorem}
Let $\mathscr{C}^{\ve} (t) = E[\dot{W}^{\ve}(t) W^{H, \ve}(t)]$.  Then, we have the following limit:
\begin{equation}
\label{eq:introRevisedWZ}
\int_0^T f(W_t^{H, \ve}) dW_t^{\ve}  - \int_0^T \mathscr{C}^{\ve}(t)f'(W_t^{H, \ve}) dt \to \int_0^T f(W_t^H) dW_t, \,\,\,\, \ve \to 0.
\end{equation}
\end{theorem}
The result is proved by the theory of regularity structures by Hairer~\cite{hairer2014}.  The following sketch illustrates the relations between the convergence of the integral and the regularity structure.  
\[
	\xymatrix{
		\Pi^{\ve} \ar@{<->}[d] \ar[r]^{R} & {\Pi}^{B, \ve} \ar@{<->}[d] \ar[r]^{\ve \to 0} & \Pi \ar@{<->}[d] \\
		\int_0^T f(W_t^{H, \ve}) dW_t^{\ve} \ar@{|->}[r] & \int_0^T f(W_t^{H, \ve}) dW_t^{\ve}  - C^{\ve} \ar[r] & \int_0^T f(W_t^H) dW_t
		}
\]
We construct a regularity structure and a model $\Pi^{\ve}$ compatible with the integral $\int_0^T f(W_t^{H, \ve})dW_t$. Because model $\Pi^{\ve}$ does not converge to any model simply , we revise it by adding some terms to obtain model $\Pi^{B, \ve}$.  This procedure $\Pi^{\ve} \mapsto \Pi^{B, \ve}$ is called renormalization, and $\Pi^{B, \ve}$ is called a renormalized model.  We can show that the renormalized model $\Pi^{B, \ve}$ converges to a model $\Pi$ in the norm of model . We can also show that the renormalized model $\Pi^{B, \ve}$ and the model $\Pi$ correspond to the revised integral $\int_0^T f(W_t^{H, \ve}) dW_t^{\ve}  - \int_0^T \mathscr{C}^{\ve}(t)f'(W_t^{H, \ve}) dt$ and the It\^{o} integral $\int_0^T f(W_t^H) dW_t$, respectively.
The renormalized model $\Pi^{B, \ve}$ can be written explicitly as
\begin{align}
	&\Pi_s^{B, \ve}({\bf 1}) = \Pi_s^{\ve}({\bf 1}), \,\, \Pi_s^{B, \ve}(\Xi) = \Pi_s^{\ve}(\Xi), \,\, \Pi_s^{B, \ve} (\I(\hat{\Xi})^n) = \Pi_s^{\ve}(\I(\hat{\Xi})^n), \nonumber\\
	\label{eq:introBayervol}
	&\Pi_s^{B, \ve} (\Xi \I(\hat{\Xi})^n) = \Pi_s^{\ve} (\Xi \I(\Xi)^n) - n \mathscr{C}^{\ve}(t) \Pi_s^{\ve}(\Xi \I(\Xi)^{n-1}), \,\,\,\, n \geq 1.
\end{align}
This renormalization stems from probabilistic and analytic results such as a decomposition by the Wick product and a transformation into the Skorokhod integral; see, for example, \cite[Lemma 3.9]{RoughVolRS} and \cite[Lemma 3.11]{RoughVolRS}.
\par
Meanwhile, Bruned et~al.~\cite{AlgRenormRS} demonstrated an algebraic procedure for renormalization in regularity structures, known as Bogoliubov--Parasiuk--Hepp--Zimmermann (BPHZ) renormalization, and Chandra and Hairer~\cite{chandra2016analytic} showed that a model subjected to BPHZ renormalization converges to some model under desired conditions.  Thus, BPHZ renormalization is key in the theory of regularity structures.  
Also, study of the relation between rough paths and regularity structures gave rise to discussion on the renormalization of rough paths, given that they are a counterpart of regularity structures (\cite{bruned2020renormalisation}, \cite{bruned2016examples}, \cite{RPRenorm}).  
\par
To supplement the current trend in the theory of regularity structures, herein we focus on BPHZ renormalization of a regularity structure corresponding to a rough path generated by Gaussian processes. We also consider an application of BPHZ renormalization to the Wong--Zakai approximation of integral (\ref{eq:introRough}).  The central theorem in this article is as follows, which corresponds to Theorem~\ref{thm:BPHZplain}.
\begin{theorem}
\label{thm:introplain}
Let symbols $\Xi_1, \ldots, \Xi_d$ correspond to Gaussian processes $\xi_1, \ldots, \xi_d$, respectively.  Let $\Pi$ be a model expressing the rough path generated by $\xi_1, \ldots, \xi_d$, and let $\hat{\Pi}$ be the BPHZ-renormalized model of $\Pi$.  Assume that for any $a_1, \ldots, a_d, b_1, \ldots, b_d \in \mathbb{R}$, $\sum_{i=1}^d a_i \dot{\xi}_i(0) + \sum_{j=1}^d b_j(0) \xi_j$ is a normal distribution with zero mean.  
For $1 \leq i \leq d$ and $n \geq 1$, we have
\begin{equation}
	\hat{\Pi}_z({\bf 1}) = \Pi_z ({\bf 1}), \,\,\,\, \hat{\Pi}_z (\Xi) = \Pi_z (\Xi), \,\,\,\, \hat{\Pi}_z (\I(\Xi_i)^n) = \Pi_z(\I(\Xi_i)^n).
\end{equation}
In addition, if $| \Xi_i \I(\Xi_j)^n| < 0$ for $1 \leq i ,j \leq d$, and $n \in \mathbb{N}$, then we have
\begin{equation}
	\hat{\Pi}_z (\Xi_i \I(\Xi_j)^n) = \Pi_z (\Xi_i \I(\Xi_j)^n) - n E[\dot{\xi}_i(0) \xi_j(0)] \Pi_z (\I(\Xi_j)^{n-1}).
\end{equation}
\end{theorem}
The result of Theorem~\ref{thm:introplain} demonstrates a plain and handy expression of BPHZ renormalization under the assumption of a rough path generated by Gaussian processes, while Bruned et~al.~\cite{AlgRenormRS} provided BPHZ renormalization in a more general setting.
We apply Theorem~\ref{thm:introplain} to the Wong--Zakai approximation, leading to the following theorem that corresponds to Theorem~\ref{thm:BPHZroughvol}.
\begin{theorem}
The processes $\xi^{\ve}$ and $\hat{\xi}^{\ve}$ correspond to Brownian motion and fractional Brownian motion convoluted by a mollifier, respectively.  Then, we have
\begin{align}
&\hat{\Pi}_z^{\ve} {\bf 1} = \Pi_z^{\ve}({\bf 1}), \,\, \hat{\Pi}_z^{\ve} \Xi = \Pi_z^{\ve}(\Xi), \,\, \hat{\Pi}_z^{\ve} (\I(\hxi)^n) = \Pi_z^{\ve}(\I(\hxi)^n), \nonumber \\
	\label{eq:introBPHZvol}
	&\hat{\Pi}_z^{\ve} \Xi \I(\hxi)^n = \Pi_z^{\ve} \Xi \I(\hxi)^n - nE[\dot{\xi}^{\ve}(0) \hat{\xi}^{\ve}(0)] \Pi_z^{\ve} \Xi \I(\hxi)^{n-1}, \,\, n \geq 1.
\end{align}
\end{theorem}
Comparing (\ref{eq:introBayervol}) with (\ref{eq:introBPHZvol}), we realize that BPHZ renormalization is mostly similar to the counterpart given by Bayer et~al.~\cite{RoughVolRS}. In addition, Theorem~\ref{thm:normconv} shows that the BPHZ-renormalized model $\hat{\Pi}^{\ve}$ converges to a naive model $\Pi$, leading to BPHZ renormalization also being effective in the Wong--Zakai approximation in integral (\ref{eq:introRough}).
\par
This article is organized as follows.  In Section~2, we discuss the construction of a model on the regularity structure and define the BPHZ renormalization of that model.  In particular, we provide a simple expression of the BPHZ-renormalized model in symbols corresponding to a geometric rough path, provided that the noises are Gaussian.  In Section~3, we consider applying these results to a rough volatility model in a financial market, and we demonstrate the correspondence between the renormalization in \cite{RoughVolRS} and BPHZ renormalization.

\section{Regularity Structure of Gaussian Rough Paths}
\label{sec:RS}
In this section, we construct a regularity structure of Gaussian rough paths, which reflects a regularity structure expressing integral (\ref{eq:introRough}). We also consider BPHZ renormalization in the regularity structure and provide a plain expression of the renormalization.
\par
Let $\bar{\xi}_i \,\, (1\leq i \leq d)$ be a stochastic process on $\mathbb{R}$ such that it has $\alpha_i$-continuity almost surely, where $\alpha_i \in (0,1)$.  Also, let $\xi_i^{\ve} :=  \varrho_{\ve}\ast \bar{\xi}_i$, where $\varrho_{\ve} = \varrho(\cdot / \ve) / \ve, \,\, \ve >0$ and $\varrho$ is a smoothing mollifier.  
We denote $\xi_i^{\ve}$ simply by $\xi_i$ unless $\ve$ must be emphasized.  
Now, we construct a regularity structure that can express the integrals
\begin{equation}
\label{eq:sec2:intgrals}
	\int_0^T \xi_i d\xi_j, \int_0^T \xi_i^2 \xi_j, \ldots, \int_0^T \xi_i^nd\xi_j, \ldots, \,\, 1 \leq i, j \leq d.
\end{equation}

\subsection{Symbols and Tree Expression}
First, we introduce symbols corresponding to the integrands  and integrations in (\ref{eq:sec2:intgrals}), which are fundamental elements in the regularity structure.
A finite set $\L$ and a map $|\cdot| : \L \to \mathbb{R}$ are defined respectively as
\begin{align}
	&\L := \{ \I , \Xi_i ; i=1, 2, \ldots, d \} ,\\
	\label{eq:sec2:deg}
	& |\I| := 1, \,\,|\Xi_i| := \alpha_i -1 . 
\end{align}

\begin{remark}
The set $\mathfrak{L}$ corresponds to the symbols that express kernels and noises, so we set $\mathfrak{L} = \{ \I, \Xi_1, \ldots, \Xi_d\}$.  Also, $\I$ corresponds to the convolution of the kernel $K$ with $\Xi_1, \ldots, \Xi_d$ to $\bar{\xi}_1, \ldots, \bar{\xi}_d$, respectively.  The degree corresponds to the regularity of the noise or kernel, so we set $|\I| =1$ and $|\Xi_i| = \alpha_i -1$ where $\alpha_i \in (0,1)$.
\end{remark}

We construct a triplet $(A, T, G)$ as follows.
\begin{itemize}
\item We define a set $S$ and a vector space $T$ as
\begin{align}
\label{eq:S}
&S:= \{ {\bf 1}, \Xi_i, \I^n, \Xi_i \I^n, \I(\Xi_i), \Xi_i \I(\Xi_j)^n; 1 \leq i, j \leq d, n \in \mathbb{N} \},\\
\nonumber
&T := \langle S \rangle,
\end{align}
where $\langle S \rangle$ is a vector space spanned by elements in $S$.
\item A set $A \subset \mathbb{R}$ is given by
\[
	A := \{ |\tau| ; \tau \in S \},
\]
where a map $|\cdot| : S \to \mathbb{R}$ is defined as (\ref{eq:sec2:deg}) and
\[
	|\tau| = \sum_{\tau'} |\tau'|,
\]
where the sum is over all $\tau' \in \mathfrak{L}$ emerging in $\tau$.  For example,
\begin{align*}
	&|\Xi_i \I(\Xi_j)| = |\Xi_i| + |\I| + |\Xi_j| = \alpha_i + \alpha_j -1,\\
	&|\Xi_i \I(\Xi_j)^n|  = |\Xi_i| + n (|\I| + |\Xi_j|) = \alpha_i + n \alpha_j -1.
\end{align*}
\item A group $G:= \{ \Gamma_h; h = (h_0, h_1, \ldots, h_d) \in \mathbb{R}^{d+1} \}$ is given by
\begin{align}
	&\Gamma_h {\bf 1} = {\bf 1}, \,\, \Gamma_h \Xi_i = \Xi_i, \,\, \Gamma_h \I(\Xi_i) = \I(\Xi_i) + h_i {\bf 1}, \Gamma_h \I = \I + h_0 {\bf 1}, \\
	&\Gamma_h (\tau \cdot \tau') = \Gamma_h \tau \cdot \Gamma_h \tau'
\end{align}
for $\tau, \tau' \in S$ such that $\tau \cdot \tau' \in S$.
\end{itemize}
We can easily show that the triplet $(A, T, G)$ is a regularity structure.
An element in $S$ can be viewed as a rooted tree by the relations shown below (see Appendix~\ref{subsec:tree} for details of rooted trees).
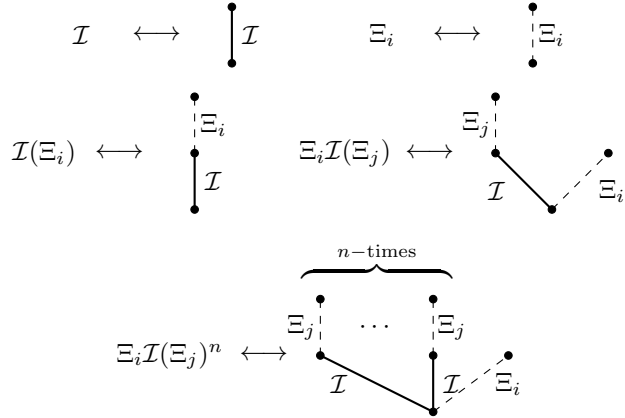
\[
\begin{tikzpicture}[scale = 0.5]
\coordinate (t_0) at (0, 1.5);
\coordinate (t_1) at (4, 0.75);
\coordinate (t_2) at (4, 2.25);

\draw [thick] (t_1) to (t_2);
\draw (t_0) node{$\I$};
\draw (2, 1.5) node{$\longleftrightarrow$};
\draw (4.5, 1.5) node{$\I$};

\filldraw (t_1) circle [radius=1mm];
\filldraw (t_2) circle [radius=1mm];

\coordinate (t1_0) at (8, 1.5);
\coordinate (t2_0) at (12, 0.75);
\coordinate (t2_1) at (12, 2.25);

\coordinate (t1_1) at (10, 1.5);

\draw [dashed] (t2_0) to (t2_1);
\draw (t1_0) node{$\Xi_i$};
\draw (t1_1) node{$\longleftrightarrow$};
\draw (12.5, 1.5) node{$\Xi_i$};

\filldraw (t2_0) circle [radius=1mm];
\filldraw (t2_1) circle [radius=1mm];
\end{tikzpicture}
\]
\[
\begin{tikzpicture}[scale =0.5]
\coordinate (t1_0) at (0, 1.5);
\coordinate (t2_0) at (4, 0);
\coordinate (t2_1) at (4, 1.5);
\coordinate (t2_2) at (4, 3);

\coordinate (t1_1) at (2, 1.5);

\draw [thick] (t2_0) to (t2_1);
\draw [dashed] (t2_1) to (t2_2);
\draw (t1_1) node{$\longleftrightarrow$};
\draw (4.5, 0.75) node{$\I$};
\draw (4.5, 2.25) node{$\Xi_i$};
\draw (t1_0) node{$\I(\Xi_i)$};

\filldraw (t2_0) circle [radius=1mm];
\filldraw (t2_1) circle [radius=1mm];
\filldraw (t2_2) circle [radius=1mm];

\coordinate (t3_0) at (8, 1.5);
\coordinate (t4_0) at (13.5, 0);
\coordinate (t4_1) at (12, 1.5);
\coordinate (t4_2) at (15, 1.5);
\coordinate (t4_3) at (12, 3);

\draw [dashed] (t4_0) to (t4_2);
\draw [thick] (t4_0) to (t4_1);
\draw [dashed] (t4_1) to (t4_3);

\filldraw (t4_0) circle [radius=1mm];
\filldraw (t4_1) circle [radius=1mm];
\filldraw (t4_2) circle [radius=1mm];
\filldraw (t4_3) circle [radius=1mm];

\draw (t3_0) node{$\Xi_i \I(\Xi_j)$};
\draw (10.25, 1.5) node{$\longleftrightarrow$};
\draw (12.5, 0.5) node[left]{$\I$};
\draw (14.5, 0.5) node[right]{$\Xi_i$};
\draw (11.5, 2.25) node{$\Xi_j$};

\end{tikzpicture}
\]
\[
\begin{tikzpicture}[scale=0.5]
\draw (-2, 1.5) node{$\Xi_i \I(\Xi_j)^n$};
\draw (0.5, 1.5) node{$\longleftrightarrow$};

\draw (1.5, 2.25) node{$\Xi_j$};
\draw (5.5, 2.25) node{$\Xi_j$};
\draw (2.5, 0.75) node{$\I$};
\draw (5.5, 0.75) node{$\I$};
\draw (7, 0.75) node{$\Xi_i$};

\draw (3.5, 4) node{$\overbrace{\hspace{20mm}}^{n\mathrm{-times}}$};
\draw (3.5, 2.25) node{$\cdots$};

\coordinate (t0) at (2, 3);
\coordinate (t1) at (2, 1.5);
\coordinate (t2) at (5, 0);
\coordinate (t3) at (5, 1.5);
\coordinate (t4) at (5, 3);
\coordinate (t5) at (7, 1.5);

\filldraw (t0) circle [radius=1mm];
\filldraw (t1) circle [radius=1mm];
\filldraw (t2) circle [radius=1mm];
\filldraw (t3) circle [radius=1mm];
\filldraw (t4) circle [radius=1mm];
\filldraw (t5) circle [radius=1mm];

\draw [thick] (t2) to (t1);
\draw [thick] (t2) to (t3);
\draw [dashed] (t2) to (t5);
\draw [dashed] (t1) to (t0);
\draw [dashed] (t3) to (t4);

\end{tikzpicture}
\]
\subsection{Model on the Regularity Structure}
\label{sec2:subsec:model}
We construct a model on the regularity structure $(A, T, G)$ to translate an abstract symbol in $S$ into a concrete function.
We define ${\bf \Pi} : T \to C^{\infty}(\mathbb{R})$ as 
\begin{align*}
	&{\bf \Pi} ({\bf 1}) = 1, \,\, {\bf \Pi} (\Xi_i)(t) = \dot{\xi}_i(t), \,\, {\bf \Pi}(\I(\Xi_i))(t) = \xi_i(t), {\bf \Pi}(\I)(t)= t, \\
	&{\bf \Pi}(\tau \cdot \tau') = {\bf \Pi}(\tau) \cdot {\bf \Pi}(\tau').
\end{align*}
Now, we define a model on $(A, T, G)$ by using the map ${\bf \Pi}$.  The pair $(\Pi, \Gamma)$ is given by the following.
\begin{itemize}
	\item For $s \in \mathbb{R}$, a map $\Pi_s: T \to \mathbb{R}$ is defined as
		\begin{align*}
		&\Pi_s {\bf 1} = 1, \,\, \Pi_s \Xi_i (t) = {\bf \Pi}(\Xi_i)(t),\\
		&\Pi_s (\I) (t) = {\bf \Pi}(\I)(t) - {\bf \Pi}(\I)(s),\\
		&\Pi_s (\I(\Xi_i))(t) = {\bf \Pi}(\I(\Xi_i))(t) - {\bf \Pi}(\I(\Xi_i))(s),\\
		&\Pi_s (\tau \cdot \tau') (t) = \Pi_s (\tau)(t) \cdot \Pi_s (\tau') (t)
		\end{align*}
		for $\tau, \tau' \in S$ such that $\tau \cdot \tau' \in S$, and the domain is extended by imposing linearity.
	\item For $s, t \in \mathbb{R}$, a map $\Gamma_{ts} : T \to T$ is defined as
		\begin{align*}
		&\Gamma_{ts} {\bf 1} = {\bf 1}, \,\, \Gamma_{ts} (\Xi_i) = \Xi_i, \\
		&\Gamma_{ts} (\I) = \I + ({\bf \Pi}(\I)(t) - {\bf \Pi}(\I(s))) {\bf 1},\\
		&\Gamma_{ts} (\I(\Xi_i)) = \I(\Xi_i) + ({\bf \Pi}(\I(\Xi_i)(t) - {\bf \Pi}(\I(\Xi_i))(s) ) {\bf 1}, \\
		&\Gamma_{ts} (\tau \cdot \tau') = \Gamma_{ts} (\tau) \cdot \Gamma_{ts}(\tau')
		\end{align*}
		for $\tau, \tau' \in S$ such that $\tau \cdot \tau' \in S$, and the domain is extended by imposing linearity.
\end{itemize}

\subsection{BPHZ Renormalization}
We revise the model $(\Pi, \Gamma)$ defined in Section~\ref{sec2:subsec:model} along the lines of the discussion in \cite{AlgRenormRS}.
First, we introduce negative renormalization in a regularity structure.
As follows, we define a set $T_-$ that plays a central role in negative renormalization:
\begin{align}
	\label{eq:S-}
	&S_- := \{ \tau_1 \sqcup \cdots \sqcup \tau_n; \tau_i \in S\setminus \{ {\bf 1} \}, n \in \mathbb{N} \} \cup \{ {\bf 1} \},\\
	\label{eq:T-}
	&\hat{T}_- := \langle S_- \rangle,
\end{align}
where $\sqcup$ is the disjoint union in set theory.  By convention, we replace $\bullet$ with $\sqcup$, which is called a forest product in \cite{AlgRenormRS}.

\begin{definition}
\begin{enumerate}
\item Let $\tau \in \hat{T}_-$.  Let $A$ be a subforest of $\tau$, that is, $A = {\bf 1}$ or $A$ consists of edges included in $\tau$.  Then, we define $R_A \tau \in \hat{T}_-$ as a forest obtained by trimming edges of $A$ from $\tau$.  $R_A \tau$ is called a contraction of $\tau$ by $A$.
\item We define a map $\Delta_-: \hat{T}_- \to \hat{T}_- \otimes \hat{T}_-$ as
\begin{equation}
	\Delta_- \tau := \sum_{A \subset \tau; A \in S_-} A \otimes R_A \tau, \,\,\,\, \tau \in S_-,
\end{equation}
where $A \subset \tau$ is a subforest of $\tau$ and $R_A \tau$ is the contraction of $\tau$ by $A$, and 
\[
	\Delta_- \left(\sum_{i=1}^n a_i \tau_i\right) := \sum_{i=1}^n a_i \Delta_- \tau_i,
\]
where $a_1, \ldots, a_n \in \mathbb{R}$ and $\tau_1, \ldots, \tau_n \in S_-$.  The restriction $\Delta_-|_T: T \to \hat{T}_- \otimes T$ denotes simply $\Delta_-$ unless confusion occurs.
\end{enumerate}
\end{definition}

\begin{example}
\begin{enumerate}
\item Let $\tau = \Xi_1 \I(\Xi_2)$, $A = \Xi_1$, and $B= \Xi_1 \bullet \Xi_2$.  Then, we have
\[
	R_A \tau = \I(\Xi_2), \,\, R_B \tau = \I
\]
because $R_A \tau$ and $R_B \tau$ are obtained by trimming the edge $\Xi_1$ and $\Xi_1 \bullet \Xi_2$, respectively.
\item We have
\begin{align*}
	\Delta_- \Xi_1 &= 1 \otimes \Xi_1 + \Xi_1 \otimes 1, \\
	\Delta_- (\I(\Xi_1)) &= 1 \otimes \I(\Xi_1) + \I \otimes \Xi_1 + \Xi_1 \otimes \I + \I(\Xi_1) \otimes 1, \\
	\Delta_- (\Xi_1\I(\Xi_1)) &= 1 \otimes \Xi_1 \I(\Xi_1) + \Xi_1 \otimes \I(\Xi_1) + \Xi_1 \otimes \Xi_1 \I + \I(\Xi_1) \otimes \Xi_1\\
	&+ \I(\Xi_1) \otimes \Xi_1+ \Xi_1 \I \otimes \Xi_1 + \Xi_1 \bullet \Xi_1 \otimes \I + \Xi_1 \I(\Xi_1).
\end{align*}
\end{enumerate}
\end{example}

\begin{definition}
\begin{enumerate}
\item A subspace $J_+ \subset T_-$ is defined as
	\[
		J_+ := \langle \{ \tau \in S_-; \tau = \sigma \bullet \bar{\sigma}, \,\, \sigma, \bar{\sigma} \in S_-, \,\, |\sigma| \geq 0 \} \rangle.
	\]
\item A quotient space $T_-$ is defined as $T_- := \hat{T}_- / J_+$, and $\mathfrak{p}^{\mathrm{ex}}_{-} : \hat{T}_{-} \to T_{-}$ denotes the canonical projection.
\item We define a map $\delexminus : T \to T_- \otimes T$ as $\delexminus := (\pexminus \otimes \mathrm{Id}) \Delta_-$.
\end{enumerate}
\end{definition}
The map $\delexminus$ takes forests including a tree with negative degree, which is ill-posed in taking a limit, from a whole tree.  BPHZ renormalization involves finding ill-posed symbols and adding extra terms there to prevent divergence, so $\delexminus$ becomes a part of BPHZ renormalization and plays a role in finding these symbols.

\begin{definition}
\label{def:BPHZ}
For the random linear map ${\bf \Pi}: T \to C^{\infty}(\mathbb{R})$ defined above, we define the following.
\begin{enumerate}
\item A character $\gminus$ on $\hat{T}_-$ is defined inductively as
\begin{align*}
&\gminus (\tau) = E[{\bf \Pi}\tau(0) ], \,\,\,\, \tau \in T ,\\
&\gminus (\tau \bullet \tau') = \gminus(\tau) \cdot \gminus(\tau'), \,\,\,\, \tau, \tau' \in \hat{T}_-.
\end{align*}
\item A twisted negative antipode $\taexminus : T_- \to \hat{T}_-$ is defined recursively as
\begin{align*}
& \taexminus 1 = 1, \\
& \taexminus \tau = - \htprominus ( \taexminus \otimes \mathrm{Id}) ( \delexminus \iexminus \tau - \tau \otimes 1), \,\,\,\, \tau \in T_- \setminus \{ 1 \},
\end{align*}
where $\htprominus$ maps $(\tau, \tau') \in \hat{T}_- \times \hat{T}_-$ to $\tau \bullet \tau' \in \hat{T}_-$, and $\iexminus : T_- \to \hat{T}_-$ is a canonical injective.
\item A pair $(\hat{\Pi}, \hat{\Gamma})$ is defined as
\begin{align}
\label{eq:pirenom}
&\hat{\Pi}_z \tau = (\gminus \taexminus \otimes \Pi_z) \delexminus \tau, \\ 
\label{eq:gamrenom}
&\hat{\Gamma}_{z\bar{z}} \tau = \Gamma_{z \bar{z}} \tau.
\end{align}
This pair is actually a model on $T$ by \cite[Theorem 6.16]{AlgRenormRS}, and we call this  the {\it BPHZ-renormalized model}.
\end{enumerate}
\end{definition}

\begin{remark}
In \cite{AlgRenormRS}, $\hat{\Gamma}_{z \bar{z}}$ is given by
\[
	\hat{\Gamma}_{z\bar{z}} \tau = (\mathrm{Id} \otimes \gamma_{z\bar{z}} (\gminus \otimes \mathrm{Id}) \delexminus) \delexplus \tau,
\]
so  it is not obvious that $\hat{\Gamma}$ is equal to $\Gamma$.  For the proof of that equality, see the Appendix~\ref{subsec:gamma}.
\end{remark}

Now, we give a plain expression of the BPHZ-renormalized model in a branched rough path, which is the central result herein.
\begin{theorem}
\label{thm:BPHZplain}
Assume that for any $a_1, \ldots, a_d, b_1, \ldots, b_d \in \mathbb{R}$, $\sum_{i=1}^d a_i \dot{\xi}_i(0) + \sum_{j=1}^d b_j\xi_j(0) $ is a normal distribution with zero mean.  
Then, for any $1 \leq i \leq d$, we have
\begin{equation}
	\label{eq:plaineq1}
	\hat{\Pi}_z (\Xi_i) = \Pi_z (\Xi_i), \,\, \hat{\Pi}_z (\I(\Xi_i)^n) = \Pi_z (\I(\Xi_i)^n).
\end{equation}
In addition, if $| \Xi_i \I(\Xi_j)^n| < 0$ for $1 \leq i ,j \leq d$, and $n \in \mathbb{N}$, then we have
\begin{equation}
	\label{eq:piplain}
	\hat{\Pi}_z (\Xi_i \I(\Xi_j)^n) = \Pi_z (\Xi_i \I(\Xi_j)^n) - n E[\dot{\xi}_i(0) \xi_j(0)] \Pi_z (\I(\Xi_j)^{n-1}).
\end{equation}
\end{theorem}

Before proving this theorem, we prepare the following lemmas.

\begin{lemma}
\label{thm:wick}
Let random variables $X_1, \ldots, X_n$ be centered jointly normal variables.  Then we have
\begin{equation}
	E \left[ \prod_{i=1}^n X_i \right] = \sum \prod_k \left[ X_{i_k}X_{j_k} \right],
\end{equation}
where the sum is taken over the set of disjoint partitions $\{ i_k, j_k \}$ of $\{ 1, \ldots, n \}$.  In particular, the left-hand side of the above expression is identical to zero if $n$ is odd.
\end{lemma}

\begin{proof}
See~\cite[Theorem 1.28]{janson1997gaussian}.
\end{proof}

\begin{lemma}
\label{lem:gminus}
If $| \Xi_i \I(\Xi_j)^n| < 0$ for $1 \leq i ,j \leq d$, and $n \geq 2$, then we have
\begin{align}
\label{eq:gminusplain}
\gminus \taexminus  \left(\Xi_i \I(\Xi_j)^n \right) &= \sum_{l=0}^n \binom{n}{l} \gminus \htprominus (\taexminus \otimes \mathrm{Id}) (\Xi_i \I(\Xi_j)^l \otimes \I(\Xi_j)^{n-l}) \\
\label{eq:gminus}
&= 0.
\end{align}
\end{lemma}

\begin{proof}

First, we show (\ref{eq:gminusplain}).  Because $\gminus$ and $\taexminus$ are multiplicative in terms of the forest product $\bullet$, for $1 \leq i \leq d$ we have 
\begin{align}
	\gminus \taexminus (\Xi_i \bullet \tau) &= \gminus \taexminus (\Xi_i) \cdot \gminus \taexminus (\tau) \nonumber \\
	&= -\gminus (\Xi_i)  \cdot \gminus \taexminus (\tau) \nonumber \\
	\label{eq:gminus0}
	&= 0.
\end{align}
We denote $\delexminus  \iexminus (\Xi_i \I(\Xi_j)^n)$ by $\delexminus \iexminus (\Xi_i \I(\Xi_j)^n) = \sum_p \tau_1^{(p)} \otimes \tau_2^{(p)}$, where $\tau_1^{(p)} \in T_-$ and $\tau_2^{(p)} \in \hat{T}_-$. 
Because $|\Xi_i \I(\Xi_j)^n| = |\tau_1^{(p)}| + |\tau_2^{(p)}|$, we have the decomposition $\tau_1^{(p)} = \tau_1' \bullet \cdots \bullet \tau_m'$ such that it satisfies either of the following conditions (a) and (b):
\begin{enumerate}
\renewcommand{\labelenumi}{(\alph{enumi})}
\item for any $1 \leq s \leq m$, $\tau_s' = \Xi_i$ or $\tau_s' = \Xi_j$;
\item we have $m=1$ and $\tau_1' = \Xi_i \I(\Xi_j)^l \,\ (1 \leq l \leq n)$.
\end{enumerate}
By this result and (\ref{eq:gminus0}), we have
\begin{align*}
	\gminus \taexminus (\Xi_i \I(\Xi)^n) &= - \htprominus (\taexminus \otimes \mathrm{Id}) (\delexminus \iexminus (\Xi_i \I(\Xi_j)^n) - \Xi_i \I(\Xi_j)^n \otimes {\bf 1})  \\
	&= \sum_{l=0}^n \binom{n}{l} \gminus \htprominus (\taexminus \otimes \mathrm{Id}) (\Xi_i \I(\Xi_j)^l \otimes \I(\Xi_j)^{n-l}),
\end{align*}
which is identical to (\ref{eq:gminusplain}).
\par
We show (\ref{eq:gminus}) by induction.  We begin by considering the case of $n =2$.  
By (\ref{eq:gminusplain}), we have
\begin{align*}
\gminus \taexminus (\Xi_i \I(\Xi_j)^2) &= -2 \gminus \htprominus (\taexminus \otimes \mathrm{Id}) (\Xi_i \I(\Xi_j) \otimes \I(\Xi_j))- \gminus(\Xi_i \I(\Xi_j)^2) \\
&= 2\gminus(\Xi_i \I(\Xi_j) \bullet \I(\Xi_j)) - \gminus(\Xi_i \I(\Xi_j)^2) \\
&= 2E[\dot{\xi}_i(0) \xi_j(0)] E[\xi_j(0)] - E[ \dot{\xi}_i(0) \xi_j(0)^2].
\end{align*}
By Theorem~\ref{thm:wick}, we have $E[\xi_j(0)] =  E[ \dot{\xi}_i(0) \xi_j(0)^2] = 0$ so that $\gminus \taexminus (\Xi_i \I(\Xi_j))=0$.  
Next, we consider (\ref{eq:gminus}) with $n \leq n_0$.  Then, by (\ref{eq:gminusplain}), we have
\begin{align*}
\gminus \taexminus  \left(\Xi_i \I(\Xi_j)^{n_0 +1} \right) &=  -\sum_{l=0}^{n_0+1} \binom{n_0+1}{l} \gminus \htprominus(\taexminus \otimes \mathrm{Id}) (\Xi_i \I(\Xi_j)^l \otimes \I(\Xi_j)^{n_0+1-l})\\
&- \gminus (\Xi_i \I(\Xi_j)^{n_0+1}).
\end{align*}
Using the assumption by induction, we have
\begin{align*}
\gminus \taexminus  \left(\Xi_i \I(\Xi_j)^{n_0 +1} \right) &= -(n_0+1)(\gminus \otimes \gminus)(\taexminus \otimes \mathrm{Id})(\Xi_i \I(\Xi_j) \otimes \I(\Xi_j)^{n_0}) \\
&- \gminus (\Xi_i \I(\Xi_j)^{n_0+1}) \\
&= (n_0+1) E[\dot{\xi}_i(0) \xi_j(0)] E[\xi_j(0)^{n_0}] - E[\dot{\xi}_i(0) \xi_j(0)^{n_0+1}].
\end{align*}
If $n_0 + 1$ is even, then we have $E[\xi_j(0)^{n_0+1}] = E[\dot{\xi}_i(0) \xi_j(0)^{n_0+1}] = 0$.  If $n_0 + 1$ is odd, then by Lemma~\ref{thm:wick} we have $(n_0+1) E[\dot{\xi}_i(0) \xi_j(0)] E[\xi_j(0)^{n_0+1}] = E[\dot{\xi}_i(0) \xi_j(0)^{n_0+1}]$.  Thus, we have $\gminus \taexminus  \left(\Xi_i \I(\Xi_j)^{n_0 +1} \right) = 0$.
\end{proof}

\begin{proof}[Proof of Theorem~\ref{thm:BPHZplain}]
First, we show (\ref{eq:plaineq1}).
Because $\delexminus \Xi_i = \Xi_i \otimes {\bf 1} + {\bf 1} \otimes \Xi_i$ and $\gminus \Xi_i = E[\dot{\xi}_i(0) ] = 0$, we have
\[
	\hat{\Pi}_z (\Xi_i) = \Pi_z (\Xi_i).
\]
We calculate $\delexminus (\I(\Xi_i)^n)$.  Because there exists no tree $\tau \subset \I(\Xi_i)^n$ such that $|\tau| < 0$ and $\gminus \taexminus \tau \neq 0$, we have
\begin{align*}
	\hat{\Pi}_z (\I(\Xi_i)^n) &= (\gminus \taexminus \otimes \Pi_z) \delexminus (\I(\Xi_i)^n) \\
	&= (\gminus \taexminus \otimes \Pi_z) ({\bf 1} \otimes \I(\Xi_i)^n)\\
	&= \Pi_z (\I(\Xi_i)^n).
\end{align*}
\par
Next, we show (\ref{eq:piplain}). 
By using (\ref{eq:gminus}) in Lemma~\ref{lem:gminus}, we have
\begin{align*}
\hat{\Pi}_z(\Xi_i \I(\Xi_j)^n) &= \sum_{l=1}^n \binom{n}{l}(\gminus \taexminus \otimes \Pi_z) (\Xi_i \I(\Xi_j)^l \otimes \I(\Xi_j)^{n-l}) + \Pi_z(\Xi_i \I(\Xi_j)^n) \\
&= n(\gminus \taexminus \otimes \Pi_z) (\Xi_i \I(\Xi_j) \otimes \I(\Xi_j)^{n-1}) + \Pi_z(\Xi \I(\Xi_j)^n) \\
&= \Pi_z (\Xi_i \I(\Xi_j)^n) - n E[\dot{\xi}_i(0) \xi_j(0)] \Pi_z (\Xi_i \I(\Xi_j)^{n-1}),
\end{align*}
and thus  we deduce the conclusion.
\end{proof}

\section{Application to a Rough Volatility Model}
In this section, we provide an example of applying BHPZ renormalization to rough paths.  As discussed in \cite{RoughVolRS} and \cite{SingRP}, we set a regularity structure compatible with a rough volatility model in a financial market.  
In \cite{RoughVolRS}, the following integral is considered:
\begin{equation}
	\label{eq:roughint}
	\int_0^T f({W}_s^H) dW_s,
\end{equation}
where $W_t$ is a Brownian motion and $W_t^H$ is a fractional Brownian motion with Hurst parameter $H \in (0, 1/2]$.   Here, we define the fractional Brownian motion in Riemann--Liouville form, that is, $W_t^H:= \sqrt{2H}\int_0^t K^H(t-s) dW_s$ where $K^H(t) = |t|^{H-1/2}$.

\begin{remark}
We do not directly define the model using fractional Brownian motion $W^H$ as a driving noise because $W^H$ does not possess stationarity. Instead, as in \cite{SingRP}, we use a stationary noise $\hat{W}$ to approximate $W^H$.  
Here, the process $\hat{W}$ is defined as $\hat{W} = \hat{K}^H \ast \xi$, where $\xi$ is a white noise on $\mathbb{R}$, and the kernel $\hat{K}^H$ is smooth on $\mathbb{R} \setminus \{ 0 \} \to \mathbb{R}$ and satisfies the following properties:
\begin{enumerate}
	\item $\hat{K}^H = K^H$ on $[0, T]$ and $\supp \hat{K}^H \subset [0, 2T]$;
	\item there exists a constant $C>0$ such that for $k=0, 1, 2$, we have
	\[
		| \partial^k \hat{K}^H(u) | \leq C |\partial^k {K}^H(u) |.
	\]
\end{enumerate}
\end{remark}

\par
We construct a regularity structure $(A, T, G)$ as follows.
\begin{itemize}
	\item We define a vector space $T$ as
	\begin{align*}
		S &:= \{ \Xi, \Xi \I(\hxi), \ldots, \Xi \I(\hxi)^M, {\bf 1}, \I(\hxi), \ldots, \I(\hxi)^M \},\\
		T &:= \langle S \rangle,
	\end{align*}
	where $M := \min \{ m \in \mathbb{N}; (m+1)(H-\kappa) - 1/2 - \kappa >0\}$.
	\item We define a degree $| \cdot |$ in $T$ as
	\begin{align*}
		| \Xi \I(\hxi)^m | &:= -1/2 - \kappa + m(H-\kappa), \,\, m \geq 0;\\
		| \I(\hxi)^m | &:= m(H-\kappa), \,\, m > 0; \\
		| {\bf 1} | &:= 0,
	\end{align*}
	and we define a set $A: = \{ |\tau| ; \tau \in T \}$.
	\item We define a structure group $G := \{ \Gamma_h; h \in \mathbb{R} \}$ as
	\begin{align*}
		&\Gamma_h {\bf 1} = {\bf 1}, \,\, \Gamma_h \Xi = \Xi, \,\, \Gamma_h \I(\hxi) = \I(\hxi) + h {\bf 1}, \\
		&\Gamma_h(\tau \cdot \tau') = \Gamma_h(\tau) \Gamma_h(\tau'), \,\, \tau, \tau' \in S.
	\end{align*}
\end{itemize}

The next step is to construct the model on $(A, T, G)$.
We prepare the notation $\mathbb{W}^n : \mathbb{R}^2 \to \mathbb{R}$ defined as
\begin{align*}
	\mathbb{W}^n(s,t) &:= \int_s^t (\hat{W}(r) - \hat{W}(s))^n d\xi(r), \,\, s \leq t;\\
	\mathbb{W}^n(s,t) &:= - \sum_{i=0}^n \binom{n}{i} (\hat{W}(t) - \hat{W}(s))^i \mathbb{W}^{n-i}(t,s), \,\, s \geq t.
\end{align*}
\begin{definition}
\label{def:naive}
We set the naive model $(\Pi, \Gamma)$ in $(A, T, G)$ as
\begin{align*}
	&\Pi_z {\bf 1} = 1, \,\, \Pi_z \Xi = \xi, \,\, \Pi_z \I(\hxi)^n = (\hat{W}(\cdot) - \hat{W}(s))^n, \\
	&\Pi_z \Xi \I(\hxi)^n = \left\{ t \mapsto \frac{d}{dt} \mathbb{W}^n(s,t) \right\},
\end{align*}
and
\begin{align*}
	&\Gamma_{ts} {\bf 1} = {\bf 1}, \,\, \Gamma_{ts} \Xi = \Xi, \,\, \Gamma_{ts} \Xi \I(\hxi) = \I(\hxi) + (\hat{W}(t) - \hat{W}(s)) {\bf 1}, \\
	&\Gamma_{ts} \tau \tau' = \Gamma_{ts} \tau \Gamma_{ts} \tau', \,\, \mathrm{for} \,\, \tau, \tau' \in S \,\, \mathrm{with} \,\, \tau \tau' \in S,
\end{align*}
where $\xi$ is a white noise on $\mathbb{R}$, and $\frac{d}{dt} \mathbb{W}^n(s,t)$ is the derivative of $\mathbb{W}^n(s, t)$ in a distributional sense.
\end{definition}
Following the procedure given in Section~\ref{sec:RS}, we generate a model with smoothing noises.
We fix a mollifier $\varrho: \mathbb{R} \to \mathbb{R}$ and set $\varrho_{\ve}(\cdot) = \varrho(\cdot/ \ve) / \ve$, $\xi^{\ve} = \varrho_{\ve} \ast \xi$, and $\hat{\xi}^{\ve} = \varrho_{\ve} \ast \hat{W}$ for $\ve >0$.
\begin{definition}
We define the smooth model $(\Pi^{\ve}, \Gamma^{\ve})$ in $(A, T, G)$ as
\begin{align*}
	&\Pi_z^{\ve} {\bf 1} = 1, \,\, \Pi_z^{\ve} \Xi = \dot{\xi}^{\ve}, \,\, \Pi_z \I(\hxi)^n = (\hat{\xi}^{\ve}(\cdot) - \hat{\xi}^{\ve}(z))^n, \\
	&\Pi_z^{\ve} \Xi \I(\hxi)^n = \dot{\xi}^{\ve} (\hat{\xi}^{\ve}(\cdot) - \hat{\xi}^{\ve}(z))^n,
\end{align*}
and
\begin{align*}
&\Gamma_{ts}^{\ve} {\bf 1} = {\bf 1}, \,\, \Gamma_{ts}^{\ve} \Xi = \Xi, \,\, \Gamma_{ts}^{\ve} \Xi \I(\hxi) = \I(\hxi) + (\hat{\xi}^{\ve}(t) - \hat{\xi}^{\ve}(s)) {\bf 1}, \\
	&\Gamma_{ts}^{\ve} \tau \tau' = \Gamma_{ts}^{\ve} \tau \Gamma_{ts}^{\ve} \tau', \,\, \mathrm{for} \,\, \tau, \tau' \in S  \,\, \mathrm{with}\,\, \tau \tau' \in S.
\end{align*}
\end{definition}

Equations~(\ref{eq:pirenom}) and (\ref{eq:gamrenom}) are well-defined in each symbol $\tau \in T$ so that we can construct the BPHZ-renormalized model $(\hat{\Pi}^{\ve}, \hat{\Gamma}^{\ve})$ in $(A, T, G)$ as in Definition~\ref{def:BPHZ}.  By Theorem~\ref{thm:BPHZplain}, we have the following theorem.
\begin{theorem}
\label{thm:BPHZroughvol}
We have
\begin{align}
\label{eq:RVrenorm1}
&\hat{\Pi}_z^{\ve} {\bf 1} = 1, \,\, \hat{\Pi}_z^{\ve} \Xi = \dot{\xi}^{\ve}, \,\, \hat{\Pi}_z^{\ve} \I(\hxi)^n = (\hat{\xi}^{\ve}(\cdot) - \hat{\xi}^{\ve}(z))^n, \\
	\label{eq:RVrenorm}
	&\hat{\Pi}_z^{\ve} \Xi \I(\hxi)^n = \Pi_z^{\ve} \Xi \I(\hxi)^n - nE[\dot{\xi}^{\ve}(0) \hat{\xi}^{\ve}(0)] \Pi_z^{\ve} \Xi \I(\hxi)^{n-1},
\end{align}
and we also have $\hat{\Gamma}^{\ve} = \Gamma^{\ve}$.
\end{theorem}
\begin{proof}
By applying Theorem~\ref{thm:BPHZplain} to the case of $d=2$, $\xi_1 = \xi^{\ve}$, and $\xi_2 = \hat{\xi}^{\ve}$, we deduce the conclusion immediately.
\end{proof}
This result corresponds to the renormalized model in \cite{RoughVolRS} and \cite{SingRP}.  In addition, BPHZ renormalization enables the model to converge to the naive model.
\begin{theorem}
\label{thm:normconv}
We have the following convergence in the norm of model :
\begin{equation}
	\label{eq:convmodel}
	\opnorm{\hat{\Pi}^{\ve} - \Pi} \to 0 \,\, \mathrm{as} \,\, \ve \to 0.
\end{equation}
\end{theorem}
\begin{proof}
It is sufficient to prove the following claim: there exists $\kappa > 0$ such that
\begin{align*}
&\left| (\Pi_s(\Xi) - \hpv_s(\Xi)) (\varphi^{\lambda}_s) \right| \lesssim \lambda^{-1/2 - \kappa} \ve^{\kappa}, \\
&\left| (\Pi_s(\I(\hxi) - \hpv_s(\I(\hxi)))) (\varphi^{\lambda}_s) \right| \lesssim \lambda^{H - \kappa} \ve^{\kappa},\\
&\left| (\Pi_s(\Xi\I(\hxi)^n) - \hpv_s(\Xi\I(\hxi)^n)) (\varphi^{\lambda}_s )\right|  \lesssim \lambda^{-1/2 + nH - \kappa} \ve^{\kappa},
\end{align*}
where $\varphi_s^{\lambda}(t) := \varphi((t-s)/\lambda) / \lambda$.
The proof of these inequalities is similar to that in \cite{RoughVolRS}.
\end{proof}

\begin{remark}
In \cite{RoughVolRS}, the renormalized model $(\bar{\Pi}^{\ve}, \bar{\Gamma}^{\ve})$ satisfies the following. For $t \in [0, T]$,
\[
	\bar{\Pi}_z^{\ve} (\Xi \I(\Xi)^n) (t) = \Pi_z \Xi \I(\Xi)^n - n E[\dot{W}^{\ve}(t) W^{H, \ve}(t)] \Pi_z \Xi \I(\Xi)^{n-1}(t), 
\]
where $\dot{W}^{\ve}(t) := \frac{d}{dt} (\varrho^{\ve} \ast W)(t)$ and $W^{H, \ve}(t) = (\varrho^{\ve} \ast W^H)(t)$.
The reason for this slight difference is that the model use raw noises such as Brownian motion and fractional Brownian motion but not the white noises $\xi$ and $\hat{W} = \hat{K}^H \ast \xi$.
\end{remark}


\appendix
\section{Trees and Forests}
\label{subsec:tree}
In this article, trees and forests from graph theory are repeatedly used as symbols of regularity structures.

\begin{definition}
Let $N$ be an abstract finite set and $E$ be the subset of $N \times N$.  We define maps $s, t: E \to N$ as $e = (v_1, v_2) \mapsto s(e) =v_1, t(e) = v_2$.
\begin{enumerate}
\item We say that the pair $T = (N, E)$ is a graph.  The elements of $N$ are called vertices or nodes, and the elements of $E$ are called  edges. If we want to emphasize $T$, then denote the sets of vertices and edges as $N_T$ and $E_T$, respectively.
\item The pair $(N, E)$ is called a simple directed graph if for any $(v_1, v_2) \in E$ we have $(v_2, v_1) \notin E$.  For a simple directed graph $(N, E)$ and $(v_1, v_2) \in E$, $v_1$ is called the parent of $v_2$, and $v_2$ is called the child of $v_1$.
\item We say that $(N, E)$ is connected if for any $v, v' \in N$ there exists a sequence of edges $e_1, \ldots, e_N$ such that $s(e_1) = v$, $t(e_N) = v'$, and $t(e_i) = s(e_{i+1}) \,\, (1\leq i \leq N-1)$.
\item We say that $(N, E)$ have a cycle if there exists a set of edges $e_1, \ldots, e_n \in E$ such that one has $t(e_i) = s(e_{i+1}) \,\,(1 \leq i \leq n-1)$ and $t(e_n) = s(e_1)$.
\item Let a simple directed graph $T = (N, E)$ be connected and have no cycle.  If there exists a unique node $v \in N$ such that $t(e) \neq v$ for any $e \in E$, then the graph $(N, E)$ is called a rooted tree and the unique node $v \in N$ is called the root of $(N, E)$.  We denote by $\varrho_T$ the root of $T$.
\item Let $\mathfrak{L}$ be a finite set,  $(N, E)$ be a rooted tree, and $\Type : E \to \mathfrak{L}$ be a map.  The triplet $(N, E, \Type)$ is called a typed tree.
\item For any rooted trees $T_1, T_2$, we define $T = T_1 \cdot T_2$ as
\begin{align*}
	N_T &:= N_{T_1}' \sqcup N_{T_2}' \sqcup \{ \varrho_T \},\\
	E_T &:= E_{T_1}' \sqcup E_{T_2}' \sqcup \{ (\varrho_T, v); v \in N_{T_1} \sqcup N_{T_2}\\
	&\hspace{45mm} \mathrm{such \,\, that \,\,} (\varrho_{T_1}, v) \in E_{T_1} \mathrm{\,\, or \,\,} (\varrho_{T_2}, v) \in E_{T_2} \}
	\end{align*}
where $N_{T_i}' := N_{T_i} \setminus \{ \varrho_{T_i} \}$, $E_{T_i}' := \{ e = (v, v') \in E_{T_i}; v \neq \varrho_{T_i} \}, i=1,2$.  This operation $T_1, T_2 \mapsto T_1 \cdot T_2$ is called a tree product.
\end{enumerate}
\end{definition}
Throughout this article, we assume that a rooted tree has no node such that it has two or more parents, that is, for any rooted tree $\tau$ and any $x \in N_{\tau}$,  we have $\# (t^{-1}( \{ x\})) \leq 1$.
\begin{example}
A rooted tree can be described as placing the root on the bottom as follows:
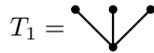
\[
\begin{tikzpicture}[scale = 0.5]
\coordinate (t1_0) at (0, 0.5);
\coordinate (t1_1) at (-1, 1.5);
\coordinate (t1_2) at (0, 1.5);
\coordinate (t1_3) at (1, 1.5);
\draw [thick] (t1_1) to (t1_0) to (t1_3);
\draw [thick] (t1_0) to (t1_2);
\filldraw (t1_0) circle [radius=1mm];
\filldraw (t1_1) circle [radius=1mm];
\filldraw (t1_2) circle [radius=1mm];
\filldraw (t1_3) circle [radius=1mm];
\draw (-2, 1) node{$T_1 =$};
\end{tikzpicture}
\]
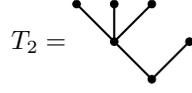
\[
\begin{tikzpicture}[scale = 0.5]
\coordinate (t2_0) at (0.5, 0);
\coordinate (t2_1) at (-0.5, 1);
\coordinate (t2_2) at (1.5, 1);
\coordinate (t2_3) at (-1.5, 2);
\coordinate (t2_4) at (-0.5, 2);
\coordinate (t2_5) at (0.5, 2);

\draw [thick] (t2_1) to (t2_0) to (t2_2);
\draw [thick] (t2_3) to (t2_1) to (t2_4);
\draw [thick] (t2_1) to (t2_5);
\draw (-2.5, 1) node{$T_2=$};

\filldraw (t2_0) circle [radius=1mm];
\filldraw (t2_1) circle [radius=1mm];
\filldraw (t2_2) circle [radius=1mm];
\filldraw (t2_3) circle [radius=1mm];
\filldraw (t2_4) circle [radius=1mm];
\filldraw (t2_5) circle [radius=1mm];
\end{tikzpicture}
\]
The tree product of $T_1$ and $T_2$ is expressed as
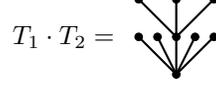
\[
\begin{tikzpicture}[scale = 0.5]
\coordinate (t1_0) at (0, 0);
\coordinate (t1_1) at (0, 1);
\coordinate (t1_2) at (0, 2);
\coordinate (t1_3) at (-1, 2);
\coordinate (t1_4) at (1, 2);
\coordinate (t1_5) at (0.5, 1);
\coordinate (t1_6) at (1, 1);
\coordinate (t1_7) at (-0.5, 1);
\coordinate (t1_8) at (-1, 1);

\draw [thick] (t1_0) to (t1_1) to (t1_2);
\draw [thick] (t1_0) to (t1_5);
\draw [thick] (t1_0) to (t1_6);
\draw [thick] (t1_0) to (t1_7);
\draw [thick] (t1_0) to (t1_8);

\draw [thick] (t1_1) to (t1_3);
\draw [thick] (t1_1) to (t1_4);
\draw (-3, 1) node{$T_1 \cdot T_2 =$};

\filldraw (t1_0) circle [radius=1mm];
\filldraw (t1_1) circle [radius=1mm];
\filldraw (t1_2) circle [radius=1mm];
\filldraw (t1_3) circle [radius=1mm];
\filldraw (t1_4) circle [radius=1mm];
\filldraw (t1_5) circle [radius=1mm];
\filldraw (t1_6) circle [radius=1mm];
\filldraw (t1_7) circle [radius=1mm];
\filldraw (t1_8) circle [radius=1mm];

\end{tikzpicture}
\]
However, the graphs $T_3$ and $T_4$ defined below are not rooted trees because $T_3$ has two candidates for the root and $T_4$ has a loop.
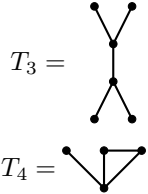
\[
\begin{tikzpicture}[scale = 0.5]
\coordinate (t1_0) at (-0.5, 0);
\coordinate (t1_1) at (0.5, 0);
\coordinate (t1_2) at (0, 1);
\coordinate (t1_3) at (0, 2);
\coordinate (t1_4) at (-0.5, 3);
\coordinate (t1_5) at (0.5, 3);

\draw [thick] (t1_0) to (t1_2) to (t1_1);
\draw [thick] (t1_2) to (t1_3) to (t1_4);
\draw [thick] (t1_3) to (t1_5);

\filldraw (t1_0) circle [radius=1mm];
\filldraw (t1_1) circle [radius=1mm];
\filldraw (t1_2) circle [radius=1mm];
\filldraw (t1_3) circle [radius=1mm];
\filldraw (t1_4) circle [radius=1mm];
\filldraw (t1_5) circle [radius=1mm];

\draw (-2, 1.5) node{$T_3 = $};
\end{tikzpicture}
\]
\[
\begin{tikzpicture}[scale = 0.5]
\coordinate (t1_0) at (0, 0);
\coordinate (t1_1) at (-1, 1);
\coordinate (t1_2) at (0, 1);
\coordinate (t1_3) at (1, 1);
\draw [thick] (t1_1) to (t1_0) to (t1_3);
\draw [thick] (t1_0) to (t1_2) to (t1_3);

\filldraw (t1_0) circle [radius=1mm];
\filldraw (t1_1) circle [radius=1mm];
\filldraw (t1_2) circle [radius=1mm];
\filldraw (t1_3) circle [radius=1mm];

\draw (-2, 0.5) node{$T_4 = $};
\end{tikzpicture}
\]

\end{example}
\begin{definition}
\begin{enumerate}
\item A graph in which each connected component is a tree is called a forest.  In addition, the forest is called a rooted forest if each connected component of it is a rooted tree.
\item Let $\mathfrak{L}$ be a finite set.  The triplet $(N, E, \Type)$ such that $(N, E)$ is a rooted forest and $\Type: E \to \mathfrak{L}$ is called a typed rooted forest.
\end{enumerate}
\end{definition}
Regarding the empty set $\emptyset$ as a forest, a vector space generated by the totality of rooted forests is a unital algebra in which the product is the disjoint union and the unit is the empty set.  This product is called a forest product.  We denote by $1$ the unit of a forest product unless confusion occurs.
\begin{definition}
A colored forest is a pair $(F, \hat{F})$ satisfying the following conditions.
\begin{enumerate}
	\item $F= (E_F, N_F, \Type)$ is a typed rooted forest.
	\item The map $\hat{F} : E_F \sqcup N_F \to \mathbb{N}\cup \{ 0 \}$ satisfies the following: for any $e = (x, y) \in E_F$ such that $\hat{F}(e) \neq 0$, we have $\hat{F}(x) = \hat{F}(y) = \hat{F}(e)$.
\end{enumerate}
$\hat{F}$ is called the coloring of $F$.  For $i > 0$, we set
\[
\hat{F}_i = (\hat{E}_i, \hat{N}_i), \,\,\,\, \hat{E}_i = \hat{F}^{-1}(i) \cap E_F, \,\,\,\, \hat{N}_i = \hat{F}^{-1}(i) \cap N_F.
\]
Finally, we denote by ${\bf CF}$ the set of all colored forests.
\end{definition}

\section{BPHZ Renormalization of $\Gamma$}
\label{subsec:gamma}
First, we introduce a positive renormalization, which is the foundation of $\Gamma$.
A map $\Delta_+: T \to T \otimes T$ is defined as
\begin{align*}
	&\Delta_+ {\bf 1} = {\bf 1} \otimes {\bf 1}, \,\, \Delta_+ \Xi_i = \Xi_i \otimes {\bf 1} + {\bf 1} \otimes \Xi_i,\\
	&\Delta_+ \I = \I \otimes {\bf 1} + {\bf 1} \otimes \I,\\ 
	&\Delta_+ \I(\Xi_i) = \I(\Xi_i) \otimes {\bf 1} + \I \otimes \Xi_i + {\bf 1} \otimes \I(\Xi_i),\\
	&\Delta_+ (\tau \cdot \tau') = \Delta_+ \tau \cdot \Delta_+ \tau'
\end{align*}
for $\tau, \tau' \in S$ such that $\tau \cdot \tau' \in S$, and the domain is extended by imposing linearity.
Note that a product on $T \otimes T$ can be defined as $(\tau_1 \otimes \tau_2) \cdot (\tau_1' \otimes \tau_2') := (\tau_1 \cdot \tau_1') \otimes (\tau_2 \cdot \tau_2')$.

\begin{definition}
\begin{enumerate}
	\item A subspace ${J}_{-}$ is defined as 
	\[
		J_- := \langle \{ \tau \in T; \tau = \sigma \cdot \bar{\sigma}, \,\, \sigma, \bar{\sigma} \in T, \,\, \sigma \neq 1_2, \,\, |\sigma| \leq 0 \} \rangle,
	\]
	where $\cdot$ is a tree product.
	\item A quotient space generated by the above subspace is defined by
	\[
		T_+ := T / J_- ,
	\]
	and $\mathfrak{p}^{\mathrm{ex}}_{G} : T \to T_+$ denotes the canonical projection.
	\item A map $\delexplus : T \to T \otimes T_+$ is defined as $\delexplus := (\mathrm{Id} \otimes \pexplus) \Delta_+$.
\end{enumerate}
\end{definition}
As defined, the vector space $T_+$ and the map $\delexplus$ can be given as
\[
	T_+ = \langle \{ {\bf 1}, \I^n, \I(\Xi_i)^n; 1 \leq i \leq d, n \in \mathbb{N} \}
\]
and
\begin{align*}
	&\delexplus {\bf 1} = {\bf 1} \otimes {\bf 1}, \,\, \delexplus \Xi_i = \Xi_i \otimes {\bf 1},\\
	&\delexplus \I = \I \otimes {\bf 1} + {\bf 1} \otimes \I,\\ 
	&\delexplus \I(\Xi_i) = \I(\Xi_i) \otimes {\bf 1} + {\bf 1} \otimes \I(\Xi_i).
\end{align*}
\begin{definition}
For $s, t \in \mathbb{R}$, a map $\gamma_{ts}: T_+ \to \mathbb{R}$ is defined as
\begin{align*}
	&\gamma_{ts} {\bf 1} = 1, \,\, \gamma_{ts} \I = {\bf \Pi}(\I)(t)-{\bf \Pi}(\I)(s),\\
	&\gamma_{ts} \I(\Xi_i) = {\bf \Pi}(\Xi_i)(t) - {\bf \Pi}(\Xi_i)(s),\\
	&\gamma_{ts} (\tau \cdot \tau') = \gamma_{ts} \tau \cdot \gamma_{ts} \tau'
\end{align*}
for $\tau, \tau' \in S$ such that $\tau \cdot \tau' \in T_+$.
\end{definition}
Note that $\Gamma_{ts} = (\mathrm{Id} \otimes \gamma_{ts}) \delexplus$.  
The BPHZ renormalization of $\Gamma$ is as follows.
\begin{definition}
Let a pair $(\Pi, \Gamma)$ be a model on the regularity structure $(A, T, G)$. The BPHZ-renormalized model $(\hat{\Pi}, \hat{\Gamma})$ is defined by
\begin{align*}
	&\hat{\Pi}_s \tau = (\gminus \taexminus \otimes \Pi_z) \delexminus \tau, \\ 
	&\hat{\Gamma}_{ts} \tau = (\mathrm{Id} \otimes \gamma_{ts} (\gminus \otimes \mathrm{Id}) \delexminus) \delexplus \tau .
\end{align*}
\end{definition}
The next proposition is the main subject of this subsection.
\begin{prop}
We have
\begin{equation}
	\hat{\Gamma}_{ts} \tau = \Gamma_{ts} \tau, \,\, \tau \in T.
\end{equation}
\end{prop}
\begin{proof}
The case where $\tau = {\bf 1}$ or $= \Xi_i$ is obvious.  We consider the case where $\tau = \I^n$.  Then, we have
\begin{align*}
	\delexplus \I^n &= (\I \otimes {\bf 1} + {\bf 1} \otimes \I)^n\\
	&= \sum_{l=0}^n \binom{n}{l} (\I^i \otimes \I^{n-l}).
\end{align*}
The edge $\I$ has positive degree $|\I| = 1$, so we have $\delexminus \I^l = {\bf 1} \otimes \I^l$.  Thus, we have
\begin{align*}
	\hat{\Gamma}_{ts} \I^n &= (\mathrm{Id} \otimes \gamma_{ts} (\gminus \otimes \mathrm{Id}) \delexminus) \sum_{l=0}^n \binom{n}{l} (\I^i \otimes \I^{n-l})\\
	&= (\mathrm{Id} \otimes \gamma_{ts} ) \sum_{l=0}^n \binom{n}{l} (\I^i \otimes \I^{n-l})\\
	&= \Gamma_{ts} \I^n.
\end{align*}
Next, we show the claim when $\tau = \Xi_i \I(\Xi_j)^n$.  Then, we have
\begin{align*}
	\delexplus  \Xi_i \I(\Xi_j)^n &= \delexplus \Xi_i \cdot \delexplus \I(\Xi_j)^n \\
	&= (\Xi_i \otimes {\bf 1}) \cdot (\I(\Xi_j) \otimes {\bf 1} + {\bf 1} \otimes \I(\Xi_j))^n\\
	&= \sum_{l=0}^n \binom{n}{l} \Xi_i \I(\Xi_j)^l \otimes \I(\Xi_j)^{n-l}.
\end{align*}
For $l \geq 1$, we have $(\gminus \taexminus \otimes \mathrm{Id}) \delexminus (\I(\Xi_j)^l) = \I(\Xi_j)^l$ because there exists no tree $\tau' \subset \I(\Xi_j)^m$ such that $|\tau'| <0$ and $\gminus \taexminus \tau' \neq 0$.  Thus, we have
\begin{align*}
	\hat{\Gamma}_{z\bar{z}} (\Xi_i \I(\Xi_j)^n) &= (\mathrm{Id} \otimes \gamma_{z \bar{z}} (\gminus \taexminus \otimes \mathrm{Id}) \delexminus) \sum_{l=0}^n \binom{n}{l} (\Xi_i \I(\Xi_j)^l \otimes \I(\Xi_j)^{n-l})\\
	&=  (\mathrm{Id} \otimes \gamma_{z \bar{z}}) \sum_{l=0}^n \binom{n}{l} (\Xi_i \I(\Xi_j)^l \otimes \I(\Xi_j)^{n-l})\\
	&= \Gamma_{z\bar{z}} (\Xi_i \I(\Xi_j)^n).
\end{align*}
\end{proof}
\bibliographystyle{abbrv}
\bibliography{manuscript}

\end{document}